\documentclass[11pt]{amsart}
\usepackage{amsmath,amsfonts,amssymb}
\usepackage{url}
\usepackage{graphicx}

\newtheorem{theorem}{\sc Theorem}

\newtheorem{proposition}[theorem]{\sc Proposition}

\newcommand{\group}{\mathcal{G}}
\newcommand{\graph}{\mathbf{G}}

\newcommand{\Aut}{\text{Aut}}

\usepackage{tikz}
\usepackage{tikz-cd}

\begin{document}
\title{Borel distinguishing number}
\author{Onur B\.{i}lge}
\address{Mutlukent Mah., Çankaya, Ankara}
\email{bilge.onur@metu.edu.tr}
\author{Burak KAYA}
\address{Department of Mathematics \\ Middle East Technical University,
 06800, Çankaya, Ankara, Turkey }
\email{burakk@metu.edu.tr}
\keywords{Descriptive graph combinatorics, distinguishing number, Borel distinguishing number}
\subjclass{03E15, 05C25}

\begin{abstract}
In this paper, we study definable variants of the notion of the distinguishing number of a graph in descriptive set theoretic setting. We introduce the notion of the Borel distinguishing number of a Borel graph and provide examples that separate distinguishing number and Borel distinguishing number at various levels. More specifically, we prove that there exist Borel graphs with countable distinguishing number but uncountable Borel distinguishing number and that, for every integer $n \geq 3$, there exists a Borel graph with distinguishing number $2$ whose Borel distinguishing number is finite and at least $n$.
\end{abstract}
\maketitle

\section{Distinguishing number and its definable variants}
Descriptive graph combinatorics is the study ``definable" graphs on Polish spaces that investigates graph theoretic properties such graphs under ``definable" constraints. For a comprehensive treatment of existing results and open questions, we refer the reader to \cite{KechrisMarks20}. So far, this study has mainly focused on chromatic numbers and perfect matchings of graphs. In this paper, we extend this study to investigate definable variants of the distinguishing number of a graph, a concept first introduced in \cite{AlbertsonCollins96}.

Let $\graph=(X,G)$ be a graph. Recall that a \textit{distinguishing coloring} of $\graph$ is a map $c: X \rightarrow I$ such that for every $\varphi \in \Aut(\graph)\setminus\{\text{id}\}$ there exists $x \in X$ such that $c(x) \neq c(\varphi(x))$. In other words, a distinguishing coloring of $\graph$ is a labeling of its vertices that results in a labeled graph with no nontrivial symmetries. Observe that a distinguishing coloring of a graph need not be a vertex coloring of the graph. The \textit{distinguishing number} of a graph $\graph=(X,G)$ is defined to be the cardinal number
\[ D(\graph)=\min\left\{|I|:\ \text{There exists a distinguishing coloring } c: X \rightarrow I\right\}\]
Suppose that $\graph=(X,G)$ is a Borel graph, i.e. $X$ is a standard Borel space and the edge relation $G \subseteq X \times X$ is a Borel subset of the product space. What would be the Borel analogue of distinguishing number?

Obviously, we first need to define the Borel analogue of a distinguishing coloring. There seem to be two natural variants depending on whether we would like to kill \textit{all symmetries} or only \textit{Borel symmetries} of the given graph.

For the first variant, we define a \textit{Borel distinguishing coloring} of $\graph$ to be a Borel map $c: X \rightarrow I$ such that for every $\varphi \in \Aut(\graph)\setminus\{\text{id}\}$ there exists $x \in X$ with $c(x) \neq c(\varphi(x))$, where $I$ is a standard Borel space.

For the second variant, we define a \textit{strictly Borel distinguishing coloring} of $\graph$ to be a Borel map $c: X \rightarrow I$ such that for every $\varphi \in \Aut_B(\graph)\setminus\{\text{id}\}$ there exists $x \in X$ with $c(x) \neq c(\varphi(x))$, where $I$ is a standard Borel space and $\Aut_B(\graph)$ is the group of Borel automorphisms of $\graph$.

Having defined these, we define the \textit{Borel distinguishing number} and the \textit{strictly Borel distinguishing number} of $\graph$ respectively to be the cardinal numbers
\begin{align*}
D_B(\graph)&=\min\left\{|I|:\ \text{There exists a Borel distinguishing } c: X \rightarrow I\right\}\\
D_{SB}(\graph)&=\min\left\{|I|:\ \text{There exists a strictly Borel distinguishing } c: X \rightarrow I\right\}
\end{align*}
It follows from the definitions that $D(\graph) \leq D_B(\graph)$ and $D_{SB}(\graph) \leq D_B(\graph)$. However, there is no relationship between $D(\graph)$ and $D_{SB}(\graph)$ holding in general. Indeed, we may have strict inequalities in both directions. For example, for a Borel graph $\graph$ with $\Aut_B(\graph)=\{\text{id}\}$ and $\Aut(\graph) \neq \{\text{id}\}$, we necessarily have $1=D_{SB}(\graph)< D(\graph)$. On the other hand, as we shall see later, it is also possible to have $D(\graph)<D_{SB}(\graph)$.

It turns out that for Borel graphs that are locally countable, i.e. the
degree of every vertex is at most countable, every strictly Borel distinguishing coloring is also a Borel distinguishing coloring and consequently, we have the following proposition.

\begin{proposition}\label{locallycountable} Let $\graph=(X,G)$ be a locally countable Borel graph. Then we have $D_{SB}(\graph) = D_B(\graph)$.
\end{proposition}

\begin{proof} Let $c: X \rightarrow I$ be a strictly Borel distinguishing coloring of $\graph$. We shall show that $c$ is also a Borel distinguishing coloring. Assume towards a contradiction that there exists $\varphi \in \Aut(\graph)\setminus\{\text{id}\}$ such that for all $x \in X$ with $c(x)=c(\varphi(x))$. As we have $\varphi \neq \text{id}$, there exists $w \in X$ with $\varphi(w) \neq w$. Set
\[S=\{\varphi^i(x): i \in \mathbb{Z},\ x \in [w]_{E_{\graph}}\}\]
where $E_{\graph}$ denotes the connectedness relation of $\graph$. Since $\graph$ is locally countable, $S$ is countable. Consider the map $\phi: X \rightarrow X$ defined by
\[\phi(x)=\begin{cases}
    x & \text{ if } x \in X \setminus S\\
    \varphi(x) & \text{ if } x \in S
\end{cases}\]
Since $S$ is a countable set that is invariant under $\varphi$, the map $\phi$ is a nontrivial Borel automorphism of $\graph$ such that $c(x)=c(\phi(x))$ for all $x \in X$, contradicting that $c$ is a strictly Borel distinguishing coloring. It now follows that $D_{SB}(\graph) \geq D_B(\graph)$ and hence $D_{SB}(\graph) = D_B(\graph)$.
\end{proof}

As a result of Proposition \ref{locallycountable}, we shall mostly be interested in $D_B(\graph)$ since most Borel graphs that naturally arise in mathematical practice are locally countable.

This paper intends to initiate the study of the Borel distinguishing number of Borel graphs. As a first step, we shall show that $D(\graph)$ and $D_B(\graph)$ can be separated at various levels. Before we conclude this section, we would like to remind the reader that, as a consequence of the Borel isomorphism theorem, we have that $D_B(\graph) \in \left\{1,2,3,\dots,\aleph_0,2^{\aleph_0}\right\}$.
\section{Some preliminaries}

In this section, we shall recall some basic definitions and results from the theory of Borel equivalence relations that will be used in the next sections. The reader who is not well-read in this topic is referred to \cite{Kanovei08} for a general treatment.

Let $X$ be a standard Borel space. An equivalence relation $E \subseteq X \times X$ on $X$ is called a \textit{Borel equivalence relation} if it is a Borel subset of $X \times X$. A Borel equivalence relation $E \subseteq X \times X$ is said to be \textit{countable} if all of its equivalence classes are countable. A countable Borel equivalence relation is said to be \textit{aperiodic} if all of its equivalence classes are infinite.

A Borel equivalence relation $E \subseteq X \times X$ is said to be \textit{smooth} if there exists some standard Borel space $Y$ and a Borel map $f: X \rightarrow Y$ such that
\[ xEy \text{ if and only if } f(x)=f(y)\]
for all $x,y \in X$. For a countable Borel equivalence relation, smoothness is equivalent to having a \textit{Borel transversal}, i.e. a Borel subset of the underlying space that intersects every equivalence class at a single point.

The Feldman-Moore theorem states that if $E \subseteq X \times X$ is a countable Borel equivalence relation, then one can find a countable discrete group $\group$ and a Borel action $\group \curvearrowright X$ such that $E$ is the orbit equivalence relation of the action $\group \curvearrowright X$. Moreover, one can find a sequence of involutions $\{g_n\}_{n \in \mathbb{N}} \subseteq \group$ such that $x E y \text{ if and only if there exists } n \in \mathbb{N}\ \ g_n \cdot x=y$, for all $x,y \in X$. For example, see \cite[Theorem 1.3]{KechrisMiller04}.

\section{A basic bound $D_B(\graph)$}

Let $\graph=(X,G)$ be a Borel graph that is locally countable. Consider the connectedness equivalence relation $E_{\graph}$ on $X$ given by
\[ x E_{\graph} y \text{ if and only if there exists a path from } x \text{ to } y\]
for all $x,y \in X$. Since $E_{\graph}$ is the union of projections of the countable Borel equivalence relations \[E_n=\{(x,y,w_1,\dots,w_n): x G w_1 G \dots G w_n G y\}\]
to the first two components, it follows from the Lusin-Novikov uniformization theorem that $E_{\graph} \subseteq X \times X$ is a countable Borel equivalence relation. We begin by noting the following result.

\begin{theorem}\label{mainproposition} Let $\graph=(X,G)$ be a locally countable Borel graph such that $E_{\graph}$ is smooth and aperiodic. Then $D_B(\graph)\leq \aleph_0$.
\end{theorem}
\begin{proof} The result is immediate if $X$ is countable, since we can color each element of $X$ to a different element of $\mathbb{N}$.

Suppose that $X$ is uncountable. Recall that the set $\mathcal{P}_{\text{inf}}(\mathbb{N})$ of infinite subsets of $\mathbb{N}$ is a $G_{\delta}$-subset of the Polish space $\mathcal{P}(\mathbb{N})$ and hence is a Polish space itself with its subspace topology. Since $E_{\graph}$ is smooth, the quotient space $X/E_{\graph}$ is an uncountable standard Borel space \cite[Proposition 6.3]{KechrisMiller04}. It now follows from the Borel isomorphism theorem that there exists a Borel isomorphism $f: X/E_{\graph} \rightarrow \mathcal{P}_{\text{inf}}(\mathbb{N})$.

Since $E_{\graph}$ is a smooth countable Borel equivalence relations, there exists a Borel transversal $T \subseteq X$ for $E_{\graph}$. Using the Feldman-Moore theorem, we obtain a Borel action $\group \curvearrowright X$ of a countable discrete $\group$ and a sequence of involutions $\{g_n\}_{n \in \mathbb{N}} \subseteq \group$ as described in Section 2. Consider the map $\varphi: X \rightarrow \mathbb{N}^+$ given by
\[ \varphi(x)=\begin{cases} \min\{n \in \mathbb{N}:\ g_n \cdot x \in T\}+2 & \text{ if } x \notin T\\ 1 & \text{ if } x \in T \end{cases}\]
Observe that $g_n$'s being involutions imply that $\varphi$ injectively assigns a positive natural number to each element on each connected component with the one picked by the transversal being the first element. Consider the Borel partial order relation $\preceq$ on $X$ given by
\[x \preceq y \text{ iff } x E_{\graph} y \text{ and }\varphi(x) \leq \varphi(y)\]
Since $\varphi$ is injective on each equivalence class and $E$ is aperiodic, the Borel relation $\preceq$ well-orders each equivalence class with order type $\omega$. Define the map $c: X \rightarrow \mathbb{N}$ by \[c(x)=\text{the } k_x^{\text{th}} \text{ element of the subset } f\left([x]_{E_{\graph}}\right)\]
where $k_x$ is such that $x$ is the $k_x$-th element of $[x]_{E_{\graph}}$ with respect to $\preceq$. It is straightforward but tedious to check that $c: X \rightarrow \mathbb{N}$ is indeed a Borel map.

Observe that the set of colors used in each connected component $[x]_{E_{\graph}}$ of $\graph$ are different from each other since $f$ is injective and every color in the set $f([x]_{E_{\graph}})$ is used for some vertex in $[x]_{E_{\graph}}$. Therefore no color preserving automorphism can move an element to a different connected component. On the other hand, since each vertex in each connected component gets a different color, no color preserving automorphism of $\graph$ can swap vertices in a single connected component. Hence $c: X \rightarrow \mathbb{N}$ is a Borel distinguishing coloring.
\end{proof}

By modifying the proof of Theorem \ref{mainproposition}, one can actually prove the following abstract general result, which we were unable to verify whether is known or not.

\begin{proposition}\label{locallykappa} Let $\kappa$ be an infinite cardinal and let $\graph=(X,G)$ be a graph of cardinality at most $2^{\kappa}$ such that each connected component is of size $\kappa$. Then $D(\graph)\leq \kappa$.
\end{proposition}

The idea, which the reader may already have extracted from the proof of Theorem \ref{mainproposition}, is as follows. There are at most $2^{\kappa}$ many connected components. Therefore, we can first choose a distinct subset of $\kappa$ of cardinality $\kappa$ for each connected component, each of which has size $\kappa$, and then color the vertices in each connected component to different colors in the relevant subset of $\kappa$. 

\section{Separating $D(\graph)$ and $D_B(\graph)$ at $\aleph_0$}

In this section, we shall provide a class examples of graphs for which we have $D(\graph) \leq \aleph_0 < D_B(\graph)$. From a graph theoretic perspective, our examples are trivial in the sense that they are all copies of disjoint union of continuum many complete graphs $K_{\mathbb{N}}$. However, due to how these complete graphs ``sit" on the underlying Polish space, we are able to separate the distinguishing number and the Borel distinguishing number. We are now ready to prove the main theorem of this section.

\begin{theorem}\label{separation} Let $E \subseteq X \times X$ be a nonsmooth aperiodic countable Borel equivalence relation on a standard Borel space $X$ and set $\graph=(X,E \setminus \Delta_X)$. Then $D(\graph) = \aleph_0 < 2^{\aleph_0} = D_B(\graph)$.
\end{theorem}

\begin{proof} It follows from Proposition \ref{locallykappa} that $D(\graph) \leq \aleph_0$. Observe that each connected component is isomorphic to $K_{\mathbb{N}}$ and consequently we must have that $D(\graph) = \aleph_0$.

Assume towards a contradiction that $D_{B}(\graph) \leq \aleph_0$. Then there exists a Borel distinguishing coloring $c: X \rightarrow \mathbb{N}$ of $\graph$. Using the Feldman-Moore theorem, we can obtain a Borel action $\group \curvearrowright X$ of a countable discrete $\group=\{g_n\}_{n \in \mathbb{N}}$ for which the connectedness relation $E_{\graph}=E \setminus \Delta_X$ is the orbit equivalence relation of the action. Consider the set
\[T=\{x \in X: \forall n \in \mathbb{N}\ c(x) \leq c(g_n \cdot x)\}\]
that is, $T$ consists of vertices that take the least color under $c: X \rightarrow \mathbb{N}$ in their respective connected component. Since the action $\group \curvearrowright X$ and $c: X \rightarrow \mathbb{N}$ are Borel maps, we have that $T$ is Borel. We claim that $T$ is a Borel transversal.

It is immediate that $T$ intersects every equivalence class. If it were that there are distinct $x,y \in T$ with $x E_{\graph} y$, then the automorphism $\varphi: X \rightarrow X$ swapping the elements $x$ and $y$ fixing the rest of the elements would be a nontrivial Borel automorphism which we know does not exist. Hence $T$ is a Borel transversal which contradicts the nonsmoothness of $E$. Hence $D_{B}(\graph)=2^{\aleph_0}$
\end{proof}

Having Theorem \ref{mainproposition} and \ref{separation} in mind, one may be tempted to think that nonsmoothness of the connectedness relation may be an obstacle to having countable distinguishing number. As shown by the next example, this turns out to be false.

\textbf{Example.} Let $S^1=\{e^{i\theta}:\theta \in \mathbb{R}\}$ and consider the graph $\mathbf{S}=(S^1,G)$ where
\[G=\left\{\left(e^{i\theta},e^{i(\theta+1)}\right),\left(e^{i(\theta+1)},e^{i\theta}\right):\theta \in \mathbb{R}\right\}\]
Observe that $\mathbf{S}$ consists of continuum many disjoint ``lines" each point of which is connected to two points obtained by rotation via $1$ radian on the unit circle. Therefore any automorphism of $\mathbf{S}$ must swap these lines with each other, possibly after reversing orientation.

Set $B=\{e^{i\theta}: \theta \in (0,\pi/2) \cup (4\pi/6,5\pi/6)\}$. Then the characteristic function $\chi_B: S^1 \rightarrow \{0,1\}$ is a Borel map. Using the density of the orbit of each point under the relevant $\mathbb{Z}$ action, it is tedious but not difficult to check that $\chi_B$ is a Borel distinguishing coloring. Hence $D_{B}(\mathbf{S})=2$.

On the other hand, the connectedness relation $E_{\mathbf{S}}$ is not smooth because any transversal $T \subseteq S^1$ of the countable Borel equivalence relation $E_{\mathbf{S}}$ has to be non-Borel as the irrational rotation $\rho: e^{i \theta} \mapsto e^{i(\theta+1)}$ preserves the Lebesgue measure on $S^1$ and $S^1 = \bigsqcup_{k \in \mathbb{Z}} \rho^k(T)$.

\section{Separating $D(\graph)$ and $D_B(\graph)$ below $\aleph_0$}

In this section, we shall provide examples of graphs for which we have \[D(\graph)< D_B(\graph)<\aleph_0\]
For the rest of this section, fix an integer $n \geq 3$. Consider the Borel action $\mathbb{Z} \curvearrowright n^{\mathbb{Z}}$ by left-shift and let $X=\text{Free}\left(n^{\mathbb{Z}}\right)$ denote the free part of this action, that is, the set of nonperiodic sequences in $n^{\mathbb{Z}}$. Let $\graph=(X,G)$ be the Borel graph on $X$ defined by
\[ \alpha\ G\ \beta \text{ if and only if } \sigma(\alpha)=\beta \text{ or } \sigma(\beta) = \alpha\]
for all $\alpha,\beta \in X$, where $\sigma$ is the left-shift map on $X$. The main objective of this section is to show that $\graph$ is as desired.

\subsection{A lower bound}
Let $k \geq 2$ be an integer and let $f: X \rightarrow k$ be an arbitrary map. Consider the map $\text{tr}_f: X \rightarrow k^\mathbb{Z}$ given by
\[\text{tr}_f(\alpha)=\left(f\left(\sigma^i(\alpha)\right)\right)_{i \in \mathbb{Z}}\]
for all $\alpha \in X$. From now on, we shall call the sequence $\text{tr}_f(\alpha)$ the color trajectory{\footnote{We would like to note that the map $\text{tr}_f$ is called the \textit{symbolic representation} of $\alpha$ with respect to the partition $\{f^{-1}(i)\}_{0 \leq i<k}$ of $X$ in the context of ergodic theory.}} of $\alpha$ with respect to the map $f$. Clearly, $\text{tr}_f\circ \sigma=\sigma \circ \text{tr}_f$.

Given a sequence $\alpha$ indexed by $\mathbb{Z}$, let $\widehat{\alpha}$ denote the sequence obtained by reflecting $\alpha$ around the index $0$, that is, $\widehat{\alpha}(i)=\alpha(-i)$ for all $i \in \mathbb{Z}$. We have the following basic but useful proposition.

\begin{proposition}\label{colortrajectoryprop} Let $f: X \rightarrow k$ be a map. Then $f$ is a Borel distinguishing coloring of $\graph$ if and only if the color trajectory $\text{tr}_f$ is injective and there exists no $\delta \in k^{\mathbb{Z}}$ such that $\delta$ and $\widehat{\delta}$ are both in the image of $\text{tr}_f$.
\end{proposition}
\begin{proof} Suppose that $f: X \rightarrow k$ is a Borel distinguishing coloring. If it were the case that $\text{tr}_f(\alpha)=\text{tr}_f(\beta)$ for some $\alpha \neq \beta \in X$, then we could define a nontrivial Borel automorphism of the relevant labeled graph by swapping the orbits \[\dots,\sigma^{-1}(\alpha),\alpha,\sigma(\alpha),\dots \text{ and } \dots,\sigma^{-1}(\beta),\beta,\sigma(\beta),\dots\] sending $\sigma^i(\alpha)$ to $\sigma^i(\beta)$ for each $i \in \mathbb{Z}$. In a similar manner, if it were that $\delta$ and $\widehat{\delta}$ are both in the image of $\text{tr}_f$, say, $\delta=\text{tr}_f(\alpha)$ and $\widehat{\delta}=\text{tr}_f(\beta)$ for some $\alpha,\beta \in X$ and $\delta \in k^{\mathbb{Z}}$, then we could define a nontrivial Borel automorphism of the relevant labeled graph by swapping the orbits
\[\dots,\sigma^{-1}(\alpha),\alpha,\sigma(\alpha),\dots \text{ and } \dots,\sigma^{-1}(\beta),\beta,\sigma(\beta),\dots\] sending $\sigma^i(\alpha)$ to $\sigma^{-i}(\beta)$ for each $i \in \mathbb{Z}$.

Now, suppose that the color trajectory $\text{tr}_f$ is injective and there exists no $\delta \in k^{\mathbb{Z}}$ such that $\delta$ and $\widehat{\delta}$ are both in the image of $\text{tr}_f$. Let $\varphi \in \Aut(\graph)$ be such that $f(\alpha)=f(\varphi(\alpha))$ for all $\alpha \in X$. Let $\alpha \in X$. Since $\varphi$ is an automorphism of $\graph$, we have $\varphi(\sigma(\alpha))=\sigma^{\pm 1}(\varphi(\alpha))$. This clearly implies that either 
\begin{align*}
\varphi(\sigma^i(\alpha))&=\sigma^i(\varphi(\alpha)) \text{ for all } i \in \mathbb{Z}\\
&\text{or}\\
\varphi(\sigma^i(\alpha))&=\sigma^{-i}(\varphi(\alpha)) \text{ for all } i \in \mathbb{Z}.    
\end{align*}
In the first case, we have
\[\text{tr}_f(\alpha)(i)=f(\sigma^i(\alpha))=f(\varphi(\sigma^i(\alpha)))=f(\sigma^i(\varphi(\alpha)))=\text{tr}_f(\varphi(\alpha))(i)\]
for all $i \in \mathbb{Z}$. In the second case, we have
\[\text{tr}_f(\alpha)(i)=f(\sigma^i(\alpha))=f(\varphi(\sigma^i(\alpha)))=f(\sigma^{-i}(\varphi(\alpha)))=\text{tr}_f(\varphi(\alpha))(-i)\]
for all $i \in \mathbb{Z}$. Thus $\text{tr}_f(\alpha)=\text{tr}_f(\varphi(\alpha))$ or $\text{tr}_f(\alpha)=\widehat{\text{tr}_f(\varphi(\alpha))}$ for all $\alpha \in X$. Since the image of $\text{tr}_f$ cannot contain both $\text{tr}_f(\varphi(\alpha))$ and $\widehat{\text{tr}_f(\varphi(\alpha))}$, the second case is impossible. But then, since $\text{tr}_f$ is injective, the first case implies that $\varphi(\alpha)=\alpha$ for all $\alpha \in X$. Thus $\varphi=\text{id}$, completing the proof that $f$ is a Borel distinguishing coloring.
\end{proof}

Combined with basic results from the entropy theory of dynamical systems, Proposition \ref{colortrajectoryprop} allows us to conclude the following.

\begin{proposition}\label{finiteseparationlowerbound} $D_{B}(\graph) \geq n$.    
\end{proposition}

\begin{proof} Let $c: X \rightarrow k$ be a Borel distinguishing coloring. By Proposition \ref{colortrajectoryprop}, the map $\text{tr}_c: X \rightarrow k^{\mathbb{Z}}$ is an injective Borel map with $\text{tr}_c\circ \sigma=\sigma \circ \text{tr}_c$. It is well-known that, unless $n \leq k$, such a map cannot exist due to various entropy arguments. Since we could not find a suitable reference for this well-known fact, we will provide a quick proof. Consider $X$ together with its standard Bernoulli measure $\mu$. Then the push forward measure $\nu(A)=\text{tr}_c^*(\mu)(A)=\mu(\text{tr}_c^{-1}(A))$ is a shift-invariant Borel probability measure on $k^{\mathbb{Z}}$. By the variational principle, the entropy of the measure preserving dynamical system $\left(k^{\mathbb{Z}},\mathcal{B}(k^{\mathbb{Z}}),\nu,\sigma\right)$ cannot exceed the topological entropy of the Bernoulli shift $\mathbb{Z} \curvearrowright k^{\mathbb{Z}}$ \cite[Theorem 6.8.1]{Downarowicz11}. It follows that $\log n \leq \log k$ and hence $n \leq k$. Since any Borel distinguishing coloring uses at least $n$ colors, we have $D_{B}(\graph) \geq n$. 
\end{proof}

\subsection{An upper bound} In this subsection, we shall construct a Borel distinguishing coloring of $\graph$.

Let $k \geq n$ and $\varphi: X \rightarrow k^{\mathbb{Z}}$ be a Borel map such that $\varphi \circ \sigma = \sigma \circ \varphi$. Consider the Borel map $f_{\varphi}: X \rightarrow k$ given by $f_{\varphi}(\alpha)=\varphi(\alpha)(0)$. Then we clearly have $\text{tr}_{f_{\varphi}}=\varphi$ and hence, it follows from Proposition \ref{colortrajectoryprop} that $f_{\varphi}$ is a Borel distinguishing coloring provided that $\varphi$ is injective and there exists no $\delta \in k^{\mathbb{Z}}$ such that $\delta$ and $\widehat{\delta}$ are both in the image of $\varphi$. In this case, the image of $\varphi$ is a shift-invariant Borel subset of $k^{\mathbb{Z}}$ that is disjoint from its image under the measure preserving map $\cdot \mapsto \widehat{\cdot}$ and hence, must be of measure $0$ with respect to its usual Bernoulli measure by the ergodicity of the Bernoulli shift $\mathbb{Z} \curvearrowright k^{\mathbb{Z}}$.

Thus the problem of constructing a Borel distinguishing coloring of $\graph$ reduces to the problem of finding such an equivariant Borel injection. While we do not claim that it is optimal, it turns out that one can do this with $k=2n-1$ colors, which is sufficient for our purposes for this specific paper. We are now ready to present the main theorem of this section.

\begin{theorem}\label{finiteseparationmaintheorem} $D_B(\graph) \leq 2n-1$.
\end{theorem}
\begin{proof} For each $1 \leq m \leq n$, let $\gamma_m \in X$ be the sequence given by
\[\gamma_m(j)=\begin{cases}
    m & \text{ if } j=0\\
    0 & \text{ if } j \neq 0\end{cases}\]
Set $W=\{\sigma^{i}(\gamma_m):\ i \in \mathbb{Z},\ 1 \leq m \leq n\}$ and
\[Y= \bigg(\{(0,\uparrow)\} \cup \big(\{1,\dots,n-1\} \times\{\uparrow,\downarrow\}\big)\bigg)^{\mathbb{Z}}\]
Since $X$ is an uncountable standard Borel space, we can fix a Borel linear ordering $<$ of $X$. Let $\varphi_1: X \setminus W \rightarrow Y$ be the Borel map given by
\[
\varphi_1(\alpha)=(\alpha(i),\text{spin}(\sigma^i(\alpha)))_{i \in \mathbb{Z}}
\]
where $\text{spin}: X \setminus W \rightarrow \{\uparrow,\downarrow\}$ is the Borel map given by
\[\text{spin}(\alpha)=\begin{cases}
\uparrow & \text{ if } \alpha(0)=0\\
\uparrow & \text{ if } \alpha(0) \neq 0 \text{ and } \alpha \geq \widehat{\alpha}\\
\downarrow & \text{ if } \alpha(0) \neq 0 \text{ and } \alpha<\widehat{\alpha}
\end{cases}\]
Let
$\varphi_2: W \rightarrow Y$ be the Borel map given by
\[\varphi_2(\sigma^i(\gamma_m))(j)=\begin{cases}
(m,\uparrow) & \text { if } i+j\geq 0\\
(m,\downarrow) & \text { if } i+j< 0
\end{cases}\]
Consider the Borel map $\varphi: X \rightarrow Y$ defined by
\[\varphi(\alpha)=\begin{cases}
    \varphi_1(\alpha) & \text{ if } \alpha \notin W\\
    \varphi_2(\alpha) & \text{ if } \alpha \in W
\end{cases}\]
It is not difficult to check that $\varphi \circ \sigma=\sigma \circ \varphi$. We claim that $\varphi$ is injective and that there exists no $\delta \in Y$ such that $\delta$ and $\widehat{\delta}$ are both in the image of $\varphi$.

Observe that the first components of the entries of sequences in the images of $\varphi_1$ and $\varphi_2$ give nonperiodic and constant sequences respectively. Consequently, in order to check the injectivity of $\varphi$, it is enough to check that $\varphi_1$ and $\varphi_2$ are injective maps. The injectivity of $\varphi_1$ follows from the injectivity of $\pi_1 \circ \varphi_1= \text{id}_{X\setminus W}$. To check the injectivity of $\varphi_2$, suppose that $\varphi_2(\sigma^i(\gamma_m))=\varphi_2(\sigma^{i'}(\gamma_{m'}))$. Then, by comparing the first components of the entries of the relevant sequences, we obtain that $m=m'$. But then, by comparing the first index at which the second components of entries of these sequences flip from $\downarrow$ to $\uparrow$ as $j$ varies from $-\infty$ to $+\infty$, we get that $i=i'$. Hence $\varphi_2$ injective. It follows that $\varphi$ injective.

That there is no $\delta \in Y$ such that $\delta$ and $\widehat{\delta}$ are both in the image of $\varphi_2$ is clear as the second components of entries of a sequence in this image are eventually $\downarrow$ and $\uparrow$ respectively, as $j$ approaches $-\infty$ and $+\infty$ respectively. Assume towards a contradiction that $\delta=\varphi_1(\alpha)$ and $\widehat{\delta}=\varphi_1(\alpha')$ for some $\alpha,\alpha' \in X \setminus W$ and $\delta \in Y$. Then, since $\pi_1 \circ \varphi_1 = \text{id}_{X \setminus W}$, we must have $\alpha'=\widehat{\alpha}$. Consequently, we have \[\widehat{\varphi_1(\alpha)}=\varphi_1(\widehat{\alpha})\]
As $\alpha \in X \setminus W$, there exists distinct $i,j \in \mathbb{Z}$ such that $\alpha(i) \neq 0$ and $\alpha(j) \neq 0$. Because $\alpha$ is not periodic, we cannot have $\sigma^i(\alpha)=\widehat{\sigma^i(\alpha)}$ and $\sigma^j(\alpha)=\widehat{\sigma^j(\alpha)}$ simultaneously. Suppose without loss of generality that $\sigma^i(\alpha) \neq \widehat{\sigma^i(\alpha)}$. Then we necessarily have
\begin{align*}
    \text{spin}(\sigma^i(\alpha)) &\neq \text{spin}(\widehat{\sigma^i(\alpha)})\\
        \text{spin}(\sigma^i(\alpha)) &\neq \text{spin}(\sigma^{-i}(\widehat{\alpha}))\\
    \pi_2\left(\varphi_1(\alpha)(i)\right) &\neq \pi_2\left(\varphi_1(\widehat{\alpha})(-i)\right)
\end{align*}
from which follows $\widehat{\varphi_1(\alpha)} \neq \varphi_1(\widehat{\alpha})$, a contradiction.

Therefore there exists no $\delta \in Y$ such that $\delta$ and $\widehat{\delta}$ are both in the image of the equivariant Borel injection $\varphi$. By the remarks at the beginning of this subsection and Proposition \ref{colortrajectoryprop}, the map $f_{\varphi}$ is a Borel distinguishing coloring with $2n-1$ colors.
\end{proof}

Let us now find the distinguishing number of $\graph$. Since $\graph$ is a disjoint union of continuum many ``lines" as an abstract graph, the distinguishing coloring in the last example of the previous section can be transferred via an appropriate bijection so that we have $D(\graph)=2$. Combining all these results, we obtain the following theorem.

\begin{theorem} For every $n \geq 3$, there exists a Borel graph $\graph$ such that
\[2=D(\graph)<n \leq D_B(\graph) < \aleph_0\]    
\end{theorem}

\section{Conclusion and further questions}

In this paper, we have given examples of Borel graphs $\graph$ for which we have separated the distinguishing number and the Borel distinguishing number at the levels
\begin{center}
\begin{tabular}{|c c |} 
 \hline
 $D(\graph)$ & $D_B(\graph)$ \\ [0.5ex] 
 \hline\hline
 countable & uncountable \\
 finite & countable\footnotemark \\
 finite & finite \\ [0.5ex] 
 \hline
\end{tabular}
\end{center}
\footnotetext{Although we have not done this explicitly, this example can easily be achieved by letting $n=\omega$ in the example of Section 5.}
Currently, we do not have an example of a Borel graph for which $D(\graph)$ is finite but $D_B(\graph)$ is uncountable. We strongly suspect that the Borel graph obtained from the free part of the shift action $\mathbb{F}_2 \curvearrowright 2^{\mathbb{F}_2}$ of the free group $\mathbb{F}_2$ on two generators provides such an example. Our suspicion is mainly due to the fact that, in order to kill all ``reflections" of the Cayley graph of $\mathbb{F}_2$, we need to mark infinitely many independent directions, contrary to the case with $\mathbb{Z}$, for which marking two directions with a spin function was sufficient in the proof of Theorem \ref{finiteseparationmaintheorem}.

We would like to pose a more general question. Let $\mathcal{G}$ be a countable discrete group with a symmetric generating set $\mathcal{S}$ with $1_{\mathcal{G}} \notin \mathcal{S}$ and let $n \geq 2$ be an integer. Consider the shift action $\mathcal{G} \curvearrowright n^{\mathcal{G}}$ given by
\[(g \cdot \alpha)(h)=\alpha\left(g^{-1}h\right)\]
Let $\text{Free}(\mathcal{G})$ denote the free part of this action. Consider the Borel graph $\text{Shift}(\mathcal{G},\mathcal{S})=(\text{Free}(\mathcal{G}),G)$ given by
\[\alpha\ G\ \beta \text{ if and only if there exists } g \in \mathcal{S} \text{ such that } g \cdot \alpha = \beta\]
\textbf{Question.} What are the values of $D_B(\text{Shift}(\mathcal{G},\mathcal{S}))$ for various countable discrete groups $\mathcal{G}$? In particular, what are the values of $D_B(\text{Shift}(\mathbb{Z},\{\pm 1\}))$ and $D_B(\text{Shift}(\mathbb{F}_2,\{a^{\pm 1},b^{\pm 1}\}))$?\\
 
\textbf{Acknowledgments.} Part of this research was done when the first author was a graduate student at Leipzig University. The first author would like to thank Lukasz Grabowski, Zoltán Vidnyánszky and Anush Tserunyan for many fruitful discussions and, especially, for pointing out the relationship between entropy and the Borel distinguishing number of shift graphs, providing a quick proof of Proposition \ref{finiteseparationlowerbound}. The second author would like to thank Tan Özalp for comments and suggestions on an earlier version of this manuscript. Some of the results in Section 3 and Section 4 were first announced by the second author at the 35$^{\text{th}}$ National Mathematics Symposium organized in Edirne, Turkey. Some of the results in Section 5 were presented by the second author at the 37$^{\text{th}}$ National Mathematics Symposium organized in Antalya, Turkey.

\bibliography{references}{}
\bibliographystyle{alpha}

\end{document}